\documentclass[12pt,twoside,reqno]{amsart}
\usepackage{amsfonts}
\usepackage{amsmath,amscd}
\usepackage{amsthm}
\usepackage{latexsym}
\usepackage{amssymb}
\usepackage{enumerate}
\usepackage{url}
\markleft{\hfill F. Diaz y Diaz\hfill E. Friedman\hfill }
\newtheorem{theorem}{Theorem}
\newtheorem*{theorem*}{Theorem}

\newtheorem{lemma}[theorem]{Lemma}

\newtheorem*{claim*}{Claim}

\newcommand{\Q}{\mathbb{Q}}
\newcommand{\R}{\mathbb{R}}

\newcommand{\N}{\mathbb{N}}

\newcommand{\e}{\mathrm{e}}

\newcommand{\ie}{{\it{i.$\,$e.\ }}}

\title{The smallest regulator  for number fields of  degree 7 with five real places}
\author{Eduardo Friedman and Gabriel Ramirez-Raposo}
\address{Departamento de Matem\'atica, Facultad de Ciencias, Universidad de Chile, \newline
Casilla 653, Santiago, Chile.} 
\address{Facultad de Matem\'aticas, Pontificia Universidad Cat\'olica de Chile,  \newline Vicu\~na Mackenna 4860, Macul, Santiago, Chile}
 \email{friedman@uchile.cl} \email{gcramirez@uc.cl}

\begin{document}

\thanks{Partially supported by   Chilean FONDECYT grant 1170176.}
\begin{abstract}  In 2016 Astudillo, Diaz y Diaz and Friedman published sharp  lower bounds for regulators  of number fields of  all signatures up to degree seven, except for fields of degree seven having  five real places.
We deal with this signature, proving that  the field with  the first discriminant   has  minimal regulator.  The new element in the proof is an extension of  Pohst's geometric method  from the totally real case to fields having  one complex place.
\end{abstract}

\maketitle

\section{Introduction}
Some thirty years ago, the number fields with smallest discriminant for  signatures up to degree seven  were all known  \cite{Od}. Recently \cite{ADF} the same was established for regulators, except that  no sharp lower bounds were proved  for one signature in degree   seven. Here we close that gap.

\begin{theorem*}
Let $k$ be a  number field  of degree seven having  five real embeddings. Then its regulator $R_k$ satisfies $R_k\ge R_{k_1}= 2.8846\ldots$, where $k_1$ is the unique field of discriminant $-2\,306\,599$ in this signature.

More precisely, except for the three unique fields  with discriminants $-2\,306\,599,$ $ -2\,369\,207$ and $-2\,616\,839$, in this signature all fields satisfy $R_k>3.2$.
\end{theorem*}
The idea in \cite{ADF} is to first use  analytic lower bounds for regulators. These are very good up to a certain value $D_{\text{anal}}(r_1,r_2)$ of the discriminant $D_k$, where $(r_1,r_2)$ is the number of (real, complex) places of $k$. Then coarse geometric bounds due to Remak \cite{Re} are used for $|D_k|\ge  D_{\text{geom}}(r_1,r_2)$. This method  works if $D_{\text{geom}}(r_1,r_2)\le D_{\text{anal}}(r_1,r_2)$, which holds for small degrees, but  fails when the  unit-rank reaches 5.

In fact,  unit-rank 5, 6 and 7 are handled in \cite{ADF}, but only for  totally real fields, where an improved inequality due to Pohst \cite{Po} is available. To deal with signature  $(5,1)$, we extend Pohst's method, allowing one of the variables to be complex.

\section{Proof} 
\noindent If  $\varepsilon $ is a unit in $k$, let
\begin{equation}\label{Mk}
m_k(\varepsilon):=\sum_{\omega} (\log\|\varepsilon\|_\omega)^2,
\end{equation}
where $\omega$ runs over  the set of archimedean places of $k$ and  $\|\!\cdot\!\|_\omega$ denotes the  corres\-ponding absolute value, normalized so that   $|\text{Norm}_{k \slash \mathbb{Q}}(a)|=\prod_{\omega \in \infty_k}\|a\|_\omega$.
A proof of the following inequality can be found in \cite[\S6]{Re} or \cite[Lemma 3.4]{Fr}.
\begin{lemma}\label{Rem}
(\textit{Remak}) Suppose $k=\Q(\varepsilon)$,   where $\varepsilon \in k$ is a unit. Then the discriminant $D_k$ satisfies
$$\log|D_k|\leq m_k(\varepsilon)A(k)+\log P_n,$$
where
$$A(k):=\sqrt{ (n^3-n-4r_2^3-2r_2)/3} , \qquad  P_n=P_n(\varepsilon_1,...,\varepsilon_n):=\prod_{1\leq i<j\leq n}\!\Big|1-\frac{\varepsilon_i}{\varepsilon_j}\Big|^2 ,$$
$n:=[k:\Q]$,  $r_2$ is the number of complex places of $k$, and the $\varepsilon_i$ are the conjugates of $\varepsilon$ arranged so that $|\varepsilon_1|\leq |\varepsilon_2|\leq\dots \leq |\varepsilon_n|$.
\end{lemma}

\begin{lemma}\label{RePo} $\mathrm{\big(Remak, Pohst\,}$\cite[(18)]{Re}\cite[Satz IV]{Po}$\mathrm{\big)}$
  If $z_1,...,z_n$ are  non-zero complex numbers arranged so that $|z_1|\le\cdots\le|z_n|$, then
\begin{equation}\label{RemIneq}
P_n(z_1,...,z_n):=\prod_{1\leq i<j\leq n} \Big|1-\frac{z_i}{z_j} \Big|^2 \leq n^n.
\end{equation}
 If, in addition, $n\le 11$ and $z_i\in\R\ (1\le i\le n)$, then
\begin{equation}\label{PoIn}
 P_n(z_1,...,z_n)   \leq 4^{\lfloor n/2\rfloor},
\end{equation}
where $ \lfloor n/2 \rfloor:=(n-1)/2 $ if $n$ is odd,   $ \lfloor n/2 \rfloor:=n/2 $ if $n$ is even.
\end{lemma}

Our main task will be to improve on Remak's bound $P_7\le7^7$ when   5 of the $z_i$'s are real and the remaining two are complex conjugates. We begin more generally, assuming henceforth that $n-2$ of  the $z_i$'s are real and the remaining two are complex conjugates.  We shall denote the real elements by $r_i\ (1\le i\le n-2)$ and the complex conjugate pair by $ x\e^{i\theta}  $ and $x\e^{-i\theta}\ (\theta\in(0,\pi),\  x>0)$, arranging them so that
\begin{equation}\label{Notation}
0<|r_1|\le |r_2|\le\cdots\le |r_{n-2}|,\qquad |r_t|\le x\le |r_{t+1}|,
\end{equation}
where if $x\ge |r_{n-2}|$ we mean $t=n-2$, and if $x\le |r_1|$ we mean $t=0$.

 Grouping the factors $|1-\frac{z_i}{z_j}|^2$ in \eqref{RemIneq} according to whether both, none or one  of $z_i,z_j\in\R$, $P_n$  factors as
\begin{equation}\label{Pncm}
 P_n=  P_{n-2}(r_1,...,r_{n-2})\cdot|1-\e^{-2i\theta}|^2  \cdot\prod_{m=1}^{n-2}|1-c_m\e^{i\theta}|^4,\ \ \ \quad c_m:=
\begin{cases}r_m/x &\text{ if }m \leq t,\\
 x/r_m &\text{ if }m > t.
\end{cases}
\end{equation}
Note that $c_m\in[-1,1], \ c_m\not=0\ \,  (1\le m\le n-2)$.

\begin{lemma}\label{RayEstimate}
If  $\,0\le c \le 1$, then
\begin{equation*}%\label{RayIneq-}
 \big|1-c \,\e^{i\theta} \big|^2\leq
\begin{cases}
1\ &\mathrm{if} \ 0\le \theta \le\pi/3,\\
2\big(1-\cos(\theta)\big) &\mathrm{if} \ \pi/3\le \theta \le\pi . \end{cases}\
\end{equation*}
 If  $\,-1\le c \le 0$, then
\begin{equation*}%\label{RayIneq+}
 \big|1-c \, \e^{i\theta} \big|^2\leq
\begin{cases}
1\ &\mathrm{if} \ 2\pi/3\le \theta \le\pi,\\
2\big(1+\cos(\theta)\big) &\mathrm{if} \ 0\le \theta \le2\pi/3. \end{cases}\
\end{equation*}
\end{lemma}

\begin{proof}
 Let $g(c):=|1-c\,\e^{i\theta}|^2=1+c^2-2c\cos(\theta)$. The critical point of $g$ is a  minimum, so we just  compare the values of $g$ at the endpoints of the intervals involved.
\end{proof}

\begin{lemma}\label{CosIneqLemma}
For $a,b >0$ and $\theta \in \mathbb{R}$,  we have
\begin{equation}\label{CosIneq}
\big(1-\cos^2(\theta)\big)^a \big(1-\cos(\theta)\big)^b \leq \frac{2^{2a+b} a^a (a+b)^{a+b}}{(2a+b)^{2a+b}}.
\end{equation}
\end{lemma}	
\begin{proof} For $-1\le x\le1$, let $g(x):=(1-x^2)^a (1-x)^b$. Elementary calculus shows that   $g$  assumes its
maximum value $M$ at $x=-b/(2a+b)$, and that $M$ is given by the right-hand side of \eqref{CosIneq}.
\end{proof}

\begin{lemma}\label{BThetaLemma} Assume $\theta\in\R$ and $-1\le  c_m\le 1$ for $1\le m\le r$. Let  $d_+$ be the number of  $c_m$ with $c_m>0$,  let $d_{-}$  be the number of   $c_m$ with $c_m<0$, and define
\begin{equation}\label{BTheta}
B_r=B_r(\theta,c_1,...,c_r):=|1-\,\e^{-2i\theta}|^2\prod_{m=1}^{r}|1-c_m\,\e^{i\theta}|^4.
\end{equation}
 Then
\begin{equation}\label{BthetaIneq}
 B_r \leq \max\!\Big(\frac{4^{2a+b} a^a(a+b)^{a+b}}{(2a+b)^{2a+b }}\,,\,\frac{4^{2+f}(1+f)^{1+f}}{(2+f)^{2+f}}\Big),
\end{equation}
where $a:=1+2\min(d_+,d_{-})$, $b:= 2|d_+ -  d_{-}|$ and $f:=2\max(d_+, d_{-})$.
\end{lemma}	
\begin{proof} Replacing $\theta$ by $-\theta$ if necessary, we can assume $0\le\theta\le\pi$.
We shall first show  that  if $\pi/3\le \theta \le2\pi/3$, then $ B_r $ is  bounded  by the first element inside the max in \eqref{BthetaIneq}.  Say  $d_+ > d_{-}$, so that $a=1+2d_{-}$ and $b=2(d_+-d_{-})$. Then,  using Lemma \ref{RayEstimate} and $\pi/3\le\theta \le2\pi/3$,
\begin{equation*}
\begin{split}
B_r & = 4\big(1-\cos^2(\theta)\big)\Big(\prod_{m=1}^r |(1-c_m e^{i\theta})|^2\Big)^2 \\
%& \le 4\big(1-\cos^2(\theta)\big)\Big(\prod_{\substack{m\\ c_m>0}} \big(2-2\cos(\theta)\big)\Big)^2  \Big(\prod_{\substack{m\\ c_m<0}}  \big(2+2\cos(\theta)\big)\Big)^2 \\ &
& \le 4\big(1-\cos^2(\theta)\big)\Big(\prod_{\substack{m\\ c_m>0}}2 \big(1-\cos(\theta)\big)\Big)^2  \Big(\prod_{\substack{m\\ c_m<0}} 2 \big(1+\cos(\theta)\big)\Big)^2 \\ & =2^{2+2(d_+ \,+\, d_{-})}\big(1-\cos^2(\theta)\big)\big(1-\cos(\theta)\big)^{  2d_+}\big(1+\cos(\theta)\big)^{ 2 d_{-}}
 \\ & = 2^{2a+b}\big(1-\cos^2(\theta)\big)^{1+ 2d_{-}}\big(1-\cos(\theta)\big)^{ 2(d_+ -  d_{-})}  \\
&=2^{2a+b} \big(1-\cos^2(\theta)\big)^a \big(1-\cos(\theta)\big)^b  \\
&\leq\frac{2^{2(2a+b)} a^a (a+b)^{a+b}}{(2a+b)^{2a+b}}  \qquad(\text{see Lemma }\ref{CosIneqLemma}),
\end{split}
\end{equation*}
proving \eqref{BthetaIneq} in this case. If   $d_+ < d_{-}$, a similar argument gives
$$
B_r \le  2^{2a+b}\big(1-\cos^2(\theta)\big)^{1+ 2d_+}\big(1+\cos(\theta)\big)^{ 2(d_{-}-d+)},
$$
and \eqref{BthetaIneq} follows as above from Lemma \ref{CosIneqLemma} (with  $\theta$ replaced  by $\theta+\pi$). The case $ d_+ = d_{-}$ is clear, since then $b=0$ and we get
$ B_r\le 2^{2a}\big(1-\cos^2(\theta)\big)^{1+ 2d_+}\le 2^{2a}$, proving \eqref{BthetaIneq} when $\pi/3\le \theta \le2\pi/3$.

If  $0\le\theta  < \pi/3$, we again use Lemmas \ref{RayEstimate} and \ref{CosIneqLemma} to get
\begin{equation*}
\begin{split}
B_r & \le 4\big(1-\cos^2(\theta)\big)  \Big(\prod_{\substack{m\\ c_m<0}}  2\big(1+\cos(\theta)\big)\Big)^2= 2^{2+2d_{-}}\big(1-\cos^2(\theta)\big) \big(1+\cos(\theta)\big)^{2d_{-}}\\
&  \le 2^{2+f}\big(1-\cos^2(\theta)\big) \big(1+\cos(\theta)\big)^f\le  2^{2+ f } \frac{2^{2+f}   (1+f)^{1+f}}{(2+f)^{2+f}}.
\end{split}
\end{equation*}
  A similar argument proves \eqref{BthetaIneq} in  the remaining case, \ie when  $ 2\pi/3<\theta \le\pi  $.
 \end{proof}

\begin{lemma}(Pohst)\label{PohstIneq} For $\alpha,\beta\in[-1,1] $, the following hold.
\begin{align*}
&(i)\ \ \ \ \mathrm{If}\ \alpha\ge 0, \ \mathrm{then}\ \ (1-\alpha)(1-\alpha \beta) \leq1.{\phantom{XXXXXXXXXXXXXXX}}\\
&(ii)\ \ \     (1-\alpha)(1-\beta)(1-\alpha \beta) \leq 2.\\
&(iii)\  \ \mathrm{If}\     |\alpha|\le|\beta|  \  \mathrm{and} \ \beta\not=0, \ \mathrm{then}\ \ \big(1-\alpha\big)\big(1-\beta\big)\big(1-(\alpha/\beta)\big) \leq 2.
\end{align*}
\end{lemma}
\begin{proof}
Inequalities $(i)$ and $(ii)$ \cite[p.\ 468]{Po} can be proved by checking for critical points and the boundary. The last  one follows from $(ii)$, on replacing  $\alpha$ by $\alpha/\beta$.
\end{proof}

We now specialize to $n=7$.
\begin{lemma}\label{P7bound} Suppose $n=7$ and $c_1>0$  in \eqref{Pncm}, then $ P_7<\e^{12}<162755.$
\end{lemma}
\noindent We note   that $7^7=823543\approx \e^{13.62}$, so we have gained a factor of a little over 5 compared with Remak's bound \eqref{RemIneq}.
\begin{proof}
We begin with \eqref{Pncm},
\begin{equation}\label{Pnfactors}
P_7=B_5P_5=B_5(\theta,c_1,...,c_5)P_5(r_1,...,r_5)\qquad\qquad\big(\text{see}\  \eqref{RemIneq}\text{ and  }\eqref{BTheta}\big).
\end{equation}
 Depending on the signs of the $c_m$, we will show that $B_5$ or $P_5$ is small.
There are 16 possibilities for the signs  of   $c_2,...,c_5$, which we divide into three cases:
\begin{enumerate}[(1)]
\item Three of the $c_m$ are of one sign and two have the opposite sign $(1\le m\le 5)$. Hence, in the notation of Lemma \ref{BThetaLemma},  $a=5$, $b=2$  and $f=6$.
\item  One of the $c_m$ is of one sign and four  have the opposite sign. Hence $a=3$, $b=6$ and $f=8$.
\item All of the $c_m$ are positive.
\end{enumerate}
In case (1), Lemma \ref{BThetaLemma} gives $B_5< 4842.63$ and Pohst's inequality \eqref{PoIn} gives $P_5\le 16$.  Now \eqref{Pnfactors} yields  $P_7< 77483$, proving the lemma in case (1).

In case (2),  Lemma \ref{BThetaLemma} only gives
\begin{equation}\label{B5case2}
B_5<  40624,
\end{equation}
 but we will improve   Pohst's bound  to $P_5\le 4$. This just suffices to prove the lemma in this case. Following Pohst \cite[p.\ 467]{Po}, for $1\le i,\ell,\ell'\leq 4$ let
$$
x_i:=\frac{ r_i}{ r_{i+1}},\qquad  \qquad  \qquad    y_{ \ell,\ell'}:=1-  \prod_{i=\ell}^{\ell'} x_i=1-\frac{r_\ell}{r_{\ell'}},
$$  and
\begin{equation*}%\label{ADef}
A=A(x_1,x_2,x_3,x_4):= \prod_{1\leq \ell \leq \ell' \le4} y_{ \ell,\ell'}  =\sqrt{P_5(r_1,...,r_5)}.
\end{equation*}
Note that $-1\le x_i\le 1,\ 0\le y_{ \ell,\ell'}\le 2$ and that the signs of the  $x_i$'s are determined from those  of the $c_m$'s and vice-versa, as we are assuming $c_1>0$ in \eqref{Pncm}. All   5 possible signs of $c_1,...,c_5$  in case (2) are shown in Table 1.

 \begin{table}[ht]
\parbox{.45\linewidth}
\centering
\caption{All sign patterns  in case (2)}
\begin{tabular}{ccccc|cccc}
$c_1$  & $c_2$ & $c_3$ & $c_4$ & $c_5$ & $x_1$ & $x_2$ & $x_3$ & $x_4$ \\
\hline+ & + & + & + & $ -$ & + & + & + & $ -$ \\
+ & $ -$  &  $ -$  & $ -$  &  $ -$ &   $ -$  & + & + & + \\
+ & + & + &  $ -$ & +      &+ & + &  $ -$  &  $ -$  \\
+ &  $ -$ & + & + & +     &  $ -$  & $ -$  & + & + \\
+ & + &  $ -$ & + & +      & + &  $ -$  &  $ -$  & + \\
\hline
\end{tabular}
\end{table}
\noindent Since    $A(x_1,x_2,x_3,x_4)=A(x_4,x_3,x_2,x_1)$, it suffices to deal with the first, middle  and last lines in Table 1.

We factor
\begin{align} \label{AA}
A & = y_{1,1}  y_{2,2}  y_{3,3} y_{4,4} y_{ 1,2}y_{ 2,3}y_{ 3,4}y_{ 1,3}y_{ 2,4}y_{ 1,4}\\
 & =(y_{1,1} y_{2,2} y_{ 1,2})(y_{3,3} y_{ 3,4})(y_{ 2,3}y_{ 2,4})(y_{ 1,3}y_{ 1,4})(y_{4,4} )
\nonumber
\end{align}
For the first line in Table 1,  $x_1,x_2,x_1x_2\geq 0$, so  we have trivially  that $y_{1,1} y_{2,2} y_{ 1,2}\leq 1$. By lemma \ref{PohstIneq} (i), using  $x_3,x_2x_3, x_1x_2x_3 \geq 0$, we have  $y_{3,3} y_{ 3,4} \leq 1$, $ y_{ 2,3}y_{ 2,4} \leq 1$ and $ y_{ 1,3}y_{ 1,4} \leq 1$. Finally   $y_{4,4}  \leq 2$, and so $A\le 2$ for the signs on the first line of Table 1.

We consider now the third line in Table 1.  Then, grouping    \eqref{AA}        differently,
\begin{equation*}
A =(y_{1,1} y_{1,4}y_{ 2,4})(y_{2,2} y_{ 2,3})(y_{ 1,2}y_{ 1,3})(y_{3,3}  y_{4,4} y_{ 3,4}).
\end{equation*}
Trivially,  $y_{1,1} y_{ 1,4}y_{ 2,4}\leq 1$. By Lemma \ref{PohstIneq} (i), since $x_2,x_1x_2 \geq 0$, we have $ y_{2,2} y_{ 2,3} \leq 1$ and $ y_{ 1,2}y_{ 1,3} \leq 1$. By Lemma \ref{PohstIneq} (ii),  $ y_{3,3} y_{4,4} y_{ 3,4}  \leq 2$, and so again $A\le 2$.

For the last  line in Table 1 we write
\begin{equation*}
A =(y_{ 1,3}y_{ 1,4}y_{ 2,4})(y_{1,1} y_{ 1,2})(y_{4,4} y_{ 3,4})(y_{2,2} y_{3,3} y_{ 2,3}).
\end{equation*}
Again  trivially,  $ y_{ 1,3}y_{ 1,4}y_{ 2,4} \leq 1$. By lemma \ref{PohstIneq} (i), since $x_1,x_4 \geq 0$, $ y_{1,1}  y_{ 1,2} \leq 1$ and $ y_{4,4}y_{ 3,4} \leq 1$. Finally,  by lemma \ref{PohstIneq} (ii), we have $y_{2,2}y_{3,3}y_{ 2,3}  \leq 2$. Thus,  in case (2) we are done proving $A\le2$, , \ie $P_5\le4$. As indicated after \eqref{B5case2}, this implies the lemma in case (2).

In  case (3) we have $ c_m>0$, and  so   $ r_m>0$ for $m=1,\dots,5$. Thus
\begin{equation}\label{yl1} 0\le 1-\dfrac{r_\ell}{r_{\ell'}}\le1 \qquad\qquad\qquad\big(  \ell<\ell' \big).
\end{equation}
We shall need
\begin{equation}\label{Rless2}
R_{\ell,\ell'}:=(1+c_\ell)(1+c_{\ell'})\big(1-(r_\ell/r_{\ell'})\big)\leq 2\qquad\qquad\qquad\big(  \ell<\ell' \big).
\end{equation}
To prove   \eqref{Rless2}, we consider three possibilities according to the position of $t$ in \eqref{Notation}. If $\ell'\le t$, then by \eqref{Pncm},  $c_\ell=r_\ell/x,\ c_{\ell'}=r_{\ell'}/x$. Hence $|c_\ell|\le |c_{\ell'}|$ and
so  Lemma \ref{PohstIneq} (iii) yields \eqref{Rless2}  (on setting $\alpha:=-c_{\ell},\ \beta:=-c_{\ell'}$). Similarly, if $\ell>t$,   $c_\ell=x/r_\ell ,\ c_{\ell'}=x/r_{\ell'}$, so $|c_{\ell'}|\le |c_\ell|$ and
  Lemma \ref{PohstIneq} (iii) yields \eqref{Rless2} (with $\alpha:=-c_{\ell'},\ \beta:=-c_\ell$). Lastly, if $\ell\le t<\ell'$,
 then   $c_\ell=r_\ell/x,\ c_{\ell'}=x/r_{\ell'}$. Now \eqref{Rless2} follows from Lemma \ref{PohstIneq} (ii).

Using \eqref{yl1} and \eqref{Rless2}, we estimate
\begin{equation*}
\begin{split}
 \sqrt{P_7}  &  = |1-e^{-2i\theta}|\ \cdot\ \prod_{1\leq\ell<\ell'\leq 5}\big(1-\frac{r_\ell}{r_{\ell'}}\big) \ \cdot\  \prod_{m=1}^5|1-c_me^{i\theta}|^2 \\ &
  \leq 2 \prod_{1\leq \ell<\ell'\leq 5}\big(1-\frac{r_\ell}{r_{\ell'}}\big)\ \cdot\ \prod_{m=1}^5 (1+c_m)^2 \\
 & = 2R_{1,2}R_{2,3}R_{3,4}R_{4,5}R_{1,5}\big(1-\frac{r_1}{r_3}\big)\big(1-\frac{r_1}{r_4}\big)\big(1-\frac{r_2}{r_4}\big)\big(1-\frac{r_2}{r_5}\big)\big(1-\frac{r_3}{r_5}\big) \\
&\le 2R_{1,2}R_{2,3}R_{3,4}R_{4,5}R_{1,5}\le 2^6.
\end{split}
\end{equation*}
Hence  $P_7\le2^{12}$. \end{proof}

We can now prove our final geometric bound.
\begin{lemma}\label{Discbound} Suppose $k$ is a number field of degree 7 having   five real places and regulator $R_k\le 3.2$. Then the discriminant $D_k$ of $k$ satisfies $\log|D_k|< 31.492 $.
\end{lemma}
 \begin{proof}
Let $\varepsilon $ yield the positive minimum value of $m_k$ in \eqref{Mk} on the units of $k$. As $[k:\Q]=7$, we have $k=\Q(\varepsilon)$.  Using the value  $\gamma_5=\sqrt[5]{8}$ for Hermite's constant in dimension 5,  we find
$
m_k\le \big(3.2\sqrt{6}\big)^{1/5}\sqrt{\gamma_5}< 1.85847$  \cite[(5)]{ADF}.
Let $r_1,...r_5$ be the five real conjugates of $\varepsilon$, ordered so that $|r_1|\le\cdots\le |r_5| $, and let $x\e^{\pm i \theta}$  be the two complex conjugates  $\big(x>0,\ \theta\in(0,\pi)\big)$. Replacing $\varepsilon$ by
$-\varepsilon$ if necessary, we may assume that $r_1>0$, so $c_1>0$ with notation  as in \eqref{Pncm}. Lemmas \ref{Rem} and \ref{P7bound} yield $\log|D_k|< 31. 4918$.  \end{proof}

We shall need the following analytic tool   \cite[Lemmas 4 and 5]{ADF}.
 \begin{lemma}\label{Analy}
Let $k$ be a number field having $r_1$ real and $r_2$ complex places, and define
$$
g(x) :=\frac{1}{2^{r_1}4\pi i}\int_{2-i\infty}^{2+i\infty}(\pi^n4^{r_2}x)^{-s\slash 2}(2s-1)\Gamma(s\slash 2)^{r_1}\Gamma(s)^{r_2} \,ds\quad(x>0,\ n:=r_1+2r_2).
$$
Suppose   $0 < d_1 \leq |D_k| \leq d_2 \leq d_3$,  and assume $g(4\slash d_3) \geq 0$. Then for any  $N \in \N$  we have  $R_k  \geq 2 G(d_1, d_2, N)$, where
\begin{equation*}
G(d_1, d_2,N) := \sum_{j=1}^N \min\!\big(g(j^{2n}\slash d_1),g(j^{2n}\slash d_2)\big).
\end{equation*}
  If the ideal class of the different of $k$ is
trivial, then  $R_k\ge4G(d_1, d_2,N)$.
\end{lemma}

We  now prove    the Theorem in \S1. So assume $(r_1,r_2)=(5,1)$ and $R_k\le 3.2$. We shall first  show that $|D_k| < 3\,030\,000 $.
Since $R_k \le 3.2$, Lemma \ref{Discbound} shows that  $|D_k|\le\e^{31.492}$. We deal separately with various subintervals of $[3\,030\,000 ,\e^{31.492}]$, always taking $d_3=\e^{31.492}$ in Lemma \ref{Analy}, noting that $g(4/\e^{31.492} )=8.5631...>0$. If  $|D_k|\le\e^{20}$, then the ideal class of the different of $k$ is
trivial \cite[Table 2]{ADF}. A calculation  shows that  $R_k\ge4G(  3\,030\,000,\e^{20},1)= 3.23...>3.2$. Hence this
range of discriminant is ruled out by   Lemma \ref{Analy}. We subdivide the remaining interval $[\e^{20},\e^{31.492}]$ into four subintervals and calculate 2$G$ for them.
\begin{align*}
&2G(\e^{31.4},\e^{31.492},3) =3.511...,\qquad &2G(\e^{31},\e^{31.4},3)=4.195...,\\
 &2G(\e^{28},\e^{31},3) =3.257...,\qquad  &2G(\e^{20},\e^{28},3) =13.295...\, .
\end{align*}
Thus,  Lemma  \ref{Analy}   rules out discriminants in the interval $[\e^{20},\e^{31.492}]$, and so $|D_k|< 3\,030\,000$. We conclude with  Table 2,  listing  $R_k$ for all fields $k$ with  $|D_k|< 3\,030\,000$  \cite{DyD}.

\begin{table}[h]\caption{All fields of degree 7 having  5 real places and $|$discriminant$|<  3\,030\,000.$}

\centering
\begin{tabular}{ccc}
%\hline
Discriminant & Polynomial &  Regulator \\
\hline\hline
$-2\,306\,599^{\phantom{\Sigma^\Sigma}} $& $x^7-3x^5-x^4+x^3+3x^2+x-1$ & 2.88465  \\ \hline
$-2\,369\,207 ^{\phantom{\Sigma^\Sigma}} $& $x^7-x^5-5x^4-x^3+5x^2+x-1$ & 2.93325 \\ \hline
$-2\,616\,839^{\phantom{\Sigma^\Sigma}} $ & $x^7-x^6-5x^5-x^4+4x^3+3x^2-x-1$ & 3.13684 \\ \hline
$-2\,790\,047^{\phantom{\Sigma^\Sigma}} $ & $x^7+x^6-2x^5-3x^4-2x^3+3x^2+4x-1$ & 3.26802 \\ \hline
$-2\,790\,551^{\phantom{\Sigma^\Sigma}} $ & $x^7-5x^5-x^4+7x^3+3x^2-3x-1$ & 3.27113 \\ \hline
$-2\,894\,039^{\phantom{\Sigma^\Sigma}} $ & $x^7-4x^5-2x^4+4x^3+4x^2-x-1$ & 3.34402 \\ \hline
$-2\,932\,823^{\phantom{\Sigma^\Sigma}} $ & $x^7-x^6-4x^3+2x^2+2x-1$ & 3.36846\\ \hline\hline
\end{tabular}
\end{table}

\end{document}